\documentclass[10pt]{article} 
\usepackage{setspace}
\usepackage{float,url}
\usepackage[utf8]{inputenc} 

\usepackage[margin=1in]{geometry}
\usepackage{graphicx} 
\usepackage{epstopdf}

\usepackage{booktabs} 
\usepackage{array} 
\usepackage{paralist} 
\usepackage{verbatim} 

\usepackage{fancyhdr} 
\usepackage{sectsty}
\allsectionsfont{\sffamily\mdseries\upshape} 

\usepackage[nottoc,notlof,notlot]{tocbibind} 
\usepackage[titles]{tocloft} 

\usepackage{amsmath, amsthm, amssymb}

\usepackage{multirow}
\usepackage{algorithm}
\usepackage{algorithmic}

\usepackage{xcolor}
\usepackage{booktabs}
\usepackage{comment}
\usepackage{amsmath}
\usepackage{mathrsfs}
\usepackage{float}
\usepackage{epstopdf}
\usepackage{multirow}
\usepackage{rotating}
\usepackage{amssymb}

\usepackage{subcaption}

\usepackage{amsthm}
\newtheorem{thm}{Theorem}
\newtheorem{lem}{Lemma}
\newtheorem{prop}{Proposition}
\newtheorem{cor}{Corollary}

\newtheorem{dfn}{Definition}

\newtheorem{exm}{Example}
\newtheorem{rem}{Remark}

\usepackage{rotating}

\usepackage{tikz,pgfplots}
\usetikzlibrary{matrix}
\usepgfplotslibrary{groupplots}
\pgfplotsset{compat=newest}
\usetikzlibrary{calc,positioning,shapes,fit,arrows,shadows}
\tikzstyle{line} = [ draw, -latex']  
\usetikzlibrary{shapes,arrows}
\usepgfplotslibrary{units}
\usepackage{url}
\usepackage{color}

\graphicspath{{figures/}} 

\usepackage{color}

\newcommand{\cH}{\mathcal{H}}

\newcommand{\cB}{\mathcal{B}}

\usepackage{mathabx}

\allowdisplaybreaks

\def \PP{ {\mathcal{P}}}

\def \SS{ {\mathcal{S}}}

\def \QQ{ {\mathcal{Q}}}
\def \PP{ {\mathcal{P}}}

\def \rr{ {\mathbb{R}}}
\def \zz{ {\mathbb{Z}}}
\def \qq{ {\mathbb{Q}}}

\def \zr{ {\mathbb{Z}^{n_1}\times\mathbb{R}^{n_2}} }

\def \rhs{b}

\def \GC{ {\textup{{\bf K}}} }

\def \IG{ {\textup{{\bf IG}}} }

\def \conv{\textup{conv}}
\def \cone{ {\textup{cone}}}

\def \int{ {\textup{int}}}
\def \rec{\textup{rec.cone}}
\def \ri{\textup{ri}}

\def \bd{ {\partial}}
\def \ls{\textup{lin.space}}

\def \tq{{\,:\,}}
\def \aff{{\textup{aff}}}

\def \valcp{{\vartheta}}
\def \valcip{{\vartheta_{MIP}}}

\def \MM{\mathbb{M}}

\def \RR{\mathbb{R}}

\newcommand{\ConeParameter}{Wideness parameter }
\newcommand{\coneParameter}{wideness parameter }
\newcommand{\coneParameterNoSpace}{wideness parameter}

\newcommand \matriz[1]{{{#1}}}

\usepgfplotslibrary{fillbetween}

\title{On the integrality gap of convex mixed-integer programs} 
\author{
	Burak Kocuk\thanks{burakkocuk@sabanciuniv.edu, Industrial Engineering Program,  Sabanc{\i} University, Istanbul, Turkey 34956, ORCID 0000-0002-4218-1116} \and
	Diego A. Mor{\'{a}}n Ram\'irez\thanks{morand@rpi.edu, Industrial and Systems Engineering Department,  Rensselaer Polytechnic Institute, Troy, NY USA 12180-3590, ORCID 0000-0002-2510-0433.}
}

\date{} 

\usepackage{color}

\def \PP{ {\mathcal{P}}}

\def \SS{ {\mathcal{S}}}

\def \rr{ {\mathbb{R}}}
\def \zz{ {\mathbb{Z}}}
\def \qq{ {\mathbb{Q}}}

\def \GC{ {\textup{{\bf K}}} }

\def \conv{\textup{conv}}
\def \cone{ {\textup{cone}}}

\def \int{ {\textup{int}}}
\def \rec{\textup{rec}}
\def \ri{\textup{ri}}

\def \bd{ {\partial}}
\def \ls{\textup{lin}}
\def \tq{{\,:\,}}
\def \aff{{\textup{aff}}}

\def \MM{\mathbb{M}}

\def \RR{\mathbb{R}}

\begin{document}
\maketitle

\begin{abstract}
We study the {\em integrality gap} of convex mixed-integer programs, that is, the difference between the optimal value of such a problem and the optimal value of its continuous relaxation. We study classes of convex sets whose associated optimization problem have finite integrality gap: Dirichlet convex sets, sets with full-dimensional recession cones and sets that can be approximated by polyhedral sets. In the latter two cases, we provide estimates for the value of the integrality gap. Finally, we study the possibility of estimating the integrality gap of nonlinear convex mixed-integer programs via rational polyhedral approximations of their feasible regions and argue that, in general, such an approach may yield arbitrarily worse bounds compared to integrality gap estimations specifically derived by studying the associated nonlinear set. 

\end{abstract}

\section{Introduction}
\subsection{Integrality gaps of convex mixed-integer programs}
 A \textit{convex mixed-integer program} (MIP) is a problem of the form
\begin{equation}
\valcip(\SS):=\inf\{\alpha^Tx\tq x\in\SS\cap(\zr)\}, \label{eq:convexIP}
\end{equation}
\noindent where $n$ is a positive integer, $n_1,n_2$ are nonnegative integers with $n=n_1+n_2$, $\SS\subseteq\rr^n$ is a nonempty convex set and  $\alpha\in\rr^n$. If all variables are constrained to be integer, that is, $n_2=0$, we say that~\eqref{eq:convexIP} is  a {\em convex integer program} (IP).

Convex MIPs have a wide range of applications (e.g., in location and inventory management \cite{Atamturk12,Kocuk2021}, 
 power distribution systems \cite{Kocuk2015,kayacik2020misocp,tuncer2022misocp}, options pricing  \cite{Pinar13}, engineering design \cite{Dai} and 
 Euclidean $k$-center problems \cite{Brandenberg}).  
In general, solving problem~\eqref{eq:convexIP} is a challenging task, however, one can obtain a lower bound on the objective function value of~\eqref{eq:convexIP} by solving the so-called  {\em continuous relaxation} 
\begin{equation}
\valcp(\SS):=\inf\{\alpha^Tx\tq x\in\SS\}, \label{eq:convexP}
\end{equation}
 which is obtained after relaxing the integrality constraints in  \eqref{eq:convexIP}. The continuous relaxation~\eqref{eq:convexP} can be efficiently solved as a convex program under mild conditions. The {\em integrality gap} of~\eqref{eq:convexIP} is
 defined as\footnote{In order to avoid the trivial case when  $\valcip(\SS)=\valcp(\SS)=-\infty$ (in which case we define $\IG(\SS)=0$), we assume for the rest of the paper that problem~\eqref{eq:convexIP} has a finite optimal value, that is, $\valcip(\SS)>-\infty$.}
$$\IG(\SS) = \valcip(\SS)-\valcp(\SS),$$
\noindent and can be interpreted as a measure of the quality of the relaxation~\eqref{eq:convexP}. Estimates for the integrality gap can be used to evaluate integer programming models (see, for instance,~\cite{temitayoOR2020}) and devise algorithms to solve some classes of integer programs (see, for instance,~\cite{IG_facilitylocation2017, IG_network_cover2000, IGscheduling2011}).

In linear mixed-integer programming, that is, when $\SS$ is a rational polyhedron, 
the most prominent proximity result is an upper bound on the distance between optimal solutions of the linear integer-program~\eqref{eq:convexIP} and its linear programming relaxation~\eqref{eq:convexP} in terms of the determinants of the square submatrices of the matrix defining the rational polyhedron  (see~\cite{blair1977value,blair1979valueII, cook1986sensitivity, paat2020distances}). In order to formally state this proximity result we need some notation. For an $m\times n$ matrix $A$ and $k\leq \min\{m,n\}$, let $\Delta_k(A)$ be the largest absolute value of a $k\times k$ submatrix of~$A$.  

\begin{thm}[\cite{cook1986sensitivity, paat2020distances}]\label{thm:cook1986sensitivity}
Consider the rational polyhedron $\SS=\{x\in\rr^n\tq \matriz{A}x\geq  \rhs\}$ with  $\matriz{A} \in \zz^{m\times n}$. Assume that $\SS\cap\zz^n\neq\emptyset$ and that $\valcp(\SS)>-\infty$. Then for each optimal solution $\hat x$ of \eqref{eq:convexP} (resp. optimal solution $x^*$ of \eqref{eq:convexIP}), there exists an optimal solution $x^*$ of \eqref{eq:convexIP} (resp. optimal solution $\hat x$ of \eqref{eq:convexP}) such that 
    $$\|\hat x-x^*\|_\infty\leq n_1 \Delta_{n-1}(A).$$
\end{thm}

 As a consequence of Theorem~\ref{thm:cook1986sensitivity}, the {\em integrality gap} of the mixed-integer program~\eqref{eq:convexIP} (with $\SS=\{x\in\rr^n\tq \matriz{A}x\geq  \rhs\}$) can be bounded above by a constant that is independent of the right-hand side $\rhs$ as 
\begin{equation}\label{eq:CookIGbound}
    \IG(\SS)\leq \|\alpha \|_1 n_1 \Delta_{n-1}(A).
\end{equation}
The bound in Theorem~\ref{thm:cook1986sensitivity} has been improved to $\|\hat x-x^*\|_\infty\leq \frac{4n+2}{9}\Delta_{n-1}(A)$ in~\cite{celaya2024proximity} for IPs with $n\ge2$, an integral r.h.s. vector  $b\in\zz^m$, an objective function defined by $\alpha\in\qq^n$, and $\hat x$ being an  optimal vertex of the associated continuous relaxation~\eqref{eq:convexP}. 
%
Recent work focus on integrality gap calculations for integer knapsack problems \cite{aliev2017integrality,karlin2011integrality} and random  linear IPs~\cite{borst2023integrality} and other related results such as using other norms to compute the distance between solutions~\cite{eisenbrand2019proximity}, obtaining sharper bounds for $A$ with row rank and $\Delta$-modularity~\cite{koeppeIGbound2026} and  improving proximity bounds using sparsity~\cite{lee2020improving}.

In nonlinear mixed-integer programming, proximity analysis is limited and, for the most part, the focus is on optimization problems with convex separable objective functions and a feasible region defined by a non-empty rational polyhedron ~\cite{hochbaum1990convex,granot1990some, werman1991relationship}. This setting can be put in the form~\eqref{eq:convexIP} through an epigraph formulation as follows: If $f(x)$ is the nonlinear function, the epigraph formulation of $\min\{f(x)\tq x\in\SS\cap(\zr)\}$ is $\min\{t\tq f(x)\leq t,\ x\in\SS\cap(\zr)\}$. We note here that the results in~\cite{hochbaum1990convex,granot1990some, werman1991relationship} explicitly use properties of
the polyhedral structure of $\SS$ in a similar manner to~\cite{cook1986sensitivity}; these techniques do not carry true to the case of general convex set $\SS$. Also see~\cite{del2022proximity}
for proximity results for an optimization problem with a concave quadratic minimization objective and~\cite{stein2016error}, where error bounds for optimality and feasibility are computed based on grid-relaxations of the feasible region. 

For a general convex set $\SS$, it could be the case that $\valcip(\SS)>-\infty$ but $\valcp(\SS)=-\infty$, so that, $\IG(\SS)=+\infty$. This situation can occur even if the set $\SS$ is a polyhedral set (of course, by Theorem~\ref{thm:cook1986sensitivity}, this cannot be a polyhedron defined by rational data).
%
%

\begin{exm}\label{SOCr}
Consider the set $\SS=\{x\in\rr^2\tq x=\lambda(1,\sqrt{2}),\ \text{for some}\ \lambda\geq0\}$ which is a ray with irrational slope emanating from the origin. Clearly, the optimal value of the integer program~\eqref{eq:convexIP} is $\inf\{-x_1-x_2\tq x\in\SS\cap\zz^2\}=0$ and the optimal value of its continuous relaxation~\eqref{eq:convexP} is $\inf\{-x_1-x_2\tq x\in\SS\}=-\infty$, and  thus $\IG(\SS)=+\infty$.
\end{exm}
%

Unfortunately, even if we restrict ourselves to (general) conic IPs defined by rational data, the integrality gap might still be infinitely large.


\begin{exm}[\cite{Moran}]
\label{exSOCr} It can be shown that the following set can be represented as a second-order conic set~\cite{BenTal01}:
\begin{align*}
\SS &= \conv ( \{x\in \rr^3\tq x_3 =0, x_1 = 0, x_2 \geq 0 \}  \cup  \{x \in \rr^3\tq x_3 = 1/2, x_2 \geq x_1^2 \}  \cup \{x \in \rr^3\tq x_3 =1, x_1 = 0, x_2 \geq 0 \}).
\end{align*}
%
%
The optimal value of the integer program~\eqref{eq:convexIP} is $\inf \{x_1\tq x \in \SS\cap\zz^3\}=0>-\infty$ and the optimal value of its continuous relaxation~\eqref{eq:convexP} is  $\inf \{x_1\tq x \in \SS\} = - \infty$. Therefore, we obtain that $\IG(\SS)=+\infty$. 
\end{exm}

\subsection{Our results}
 In this paper, we study the integrality gap  of problem~\eqref{eq:convexIP} for a general convex set  $\SS$. Our main contributions are as follows.

\begin{itemize}
\item
{\bf Finiteness of the integrality gap} (Section~\ref{sec:finiteness_of_IG}): 
 %
%
We study classes of convex sets whose associated optimization problems have finite integrality gap: Dirichlet convex sets, sets with full-dimensional recession cones and sets that can be approximated by polyhedral sets. We establish that the integrality gap of convex MIPs defined by \textit{Dirichlet convex} sets is finite (see Proposition~\ref{prop:Dirichlet_IG}) and give sufficient conditions for a set to be a Dirichlet convex set (see Proposition~\ref{prop:preimage_Dirichlet_convex_sets} and Corollary~\ref{cor:conic_sets}). 
%
We also obtain integrality gap estimations for certain convex sets. For instance, utilizing the \textit{wideness parameter} (see Definition~\ref{dfn:fatnessParameterCone}), we provide an upper bound on the integrality gap when the recession cone of the convex set $\SS$ is full-dimensional (see  Theorem~\ref{thm:proximityInFullDimCone}). Moreover, under a technical condition, we also derive an upper bound  on the integrality gap  when the convex set $\SS$ is \textit{almost rational polyhedral} (see Definition~\ref{dfn:almostrationalpoly}) by using Theorem~\ref{thm:cook1986sensitivity}.
\item 
{\bf Integrality gap via polyhedral approximations} (Section~\ref{sec:intGapViaPolyApprox}): We study the possibility of estimating the integrality gap of~\eqref{eq:convexIP} by approximating its feasible region by a rational polyhedron and then using the well-known integrality gap results from the linear MIP literature. We show that this approach may have several pitfalls: First, one may need exponentially many linear constraints in order to approximate the convex set, and these constraints must be defined by rational data in order to be able to use the known integrality gap results, and second, the bound obtained by using Theorem~\ref{thm:cook1986sensitivity} and the polyhedral approximation can be arbitrarily far off. Moreover, even if the polyhedral approximation is chosen to give the best integrality gap bound, this bound can be arbitrarily bad (see Theorem~\ref{thm:noBoundedCoeffPolytopeCanApproxBall}). We illustrate the above issues in the case of the convex set being a ball centered at the origin.
\end{itemize}

\subsection{Notation and basic facts}


For a set $X\subseteq \rr^n$, we denote its interior as $\int(X)$, its boundary as $\bd X$, its convex hull as $\conv(X)$, and its conic hull as $\cone(X)$. The {\em integer hull} of $X$ is the convex hull of the integer points contained in the set $X$, namely $\conv(X\cap \zz^n)$. For a convex set $\SS$ its recession cone is the set 
$\rec(\SS)=\{d\in \rr^n\tq x+\lambda d\in \SS\ \text{for all}\ \lambda\in \rr_+,\ x\in \SS\}$, and its lineality space is the set $\ls(\SS)=\{d\in \rr^n\tq x+\lambda d\in \SS\ \text{for all}\ \lambda\in \rr,\ x\in \SS\}$.
For $i=1,\ldots,n$, we denote $e_i$ as the $i$-th vector in the canonical basis of $\rr^n$, that is, $e_i$ is the vector with a 1 in the $i$-th component and zeros otherwise. 
The ball centered at $x_0\in\rr^n$ with radius $r>0$ is denoted by 
$B(x_0,r) = \{x\in\mathbb{R}^n \tq \|x-x_0\|_2 \le r \}$.


%
{A cone $\GC\subseteq\rr^m$ is called a regular cone if it is full-dimensional, closed, convex, and pointed (it does not contain lines).
The dual cone of a cone  $\GC$ is defined as $\GC_*=\{x\in\rr^n\tq x^Ty\geq0\ \text{for all}\ y\in \GC\}$. For a closed convex cone, it can be shown that $\GC=(\GC_*)_*$. 
}


A general mixed-integer lattice~\cite{Bertsimas_Weismantel_2005}
is a set of the form $\MM=\{x \in \rr^m \,:\, x=\matriz{E}z+\matriz{F}y,\ z\in \zz^{n_1}, y\in \rr^{n_2}\}$ for $\matriz{E}\in\qq^{m\times n_{1}}$ and $\matriz{F}\in\qq^{m\times n_{2}}$, where the columns of $\matriz{E}$ are in the orthogonal subspace to the linear subspace generated by the columns of $\matriz{F}$ (this representation is without loss of generality, see, for instance, Proposition~3.11 in~\cite{DM2016}).

The {\em covering radius} of a general mixed-integer lattice is the minimum $\rho$ such that $B(0,\rho) +\MM=\rr^m$. From this definition, it follows that if $\mu(\MM)$ is the covering radius of $\MM$, then for all $\tilde x\in\rr^n$, the ball $B(\tilde x, \mu(\MM))$ contains a point in $\MM$. It is easy to see that for the mixed-integer lattice $\zz^{n_1}\times \rr^{n_2}$, its covering radius is $\mu(\zz^{n_1}\times \rr^{n_2})=\frac{\sqrt{n_1}}2$.

For simplicity, we only consider convex mixed-integer programs in this paper.  However,  the results in this paper can be easily extended to general mixed-integer lattices by noting  that $\MM=\begin{bmatrix}\matriz{E} & \matriz{F}\end{bmatrix}(\zz^{n_1}\times \rr^{n_2})$ and that linear mappings preserve convexity and related properties.

\section{The integrality gap of general convex MIPs}\label{sec:finiteness_of_IG}

%
%

\subsection{Convex MIPs with finite integrality gap}
\label{sec:DirichletConvCase}

Note that a convex MIP defined by a convex set $\SS\subseteq\rr^n$ has finite integrality gap if and only if the following property holds: $\valcip(\SS)>-\infty$ implies that $\valcp(\SS)>-\infty$. This property has been called the `finiteness property' in~\cite{Moran,kocuk2019subadditive}. The following remark relates the finiteness property with having a finite integrality gap.

\begin{rem}\label{rem:fineteness_prop}
$\IG(\SS)<+\infty$  if and only if $\SS$ satisfies the `finiteness property'.
\end{rem}

Conditions for a set to satisfy the finiteness property have been studied in~\cite{DM2011,Moran,kocuk2019subadditive} for Dirichlet convex sets. A convex set $\SS\subseteq \rr^n$ is said to be a {\em Dirichlet convex set} with respect to $\zr$ if for all $z\in \SS\cap(\zr)$, $r\in \rec(\SS)$, $\epsilon>0$ and $\gamma\geq0$, there exists a point $w\in \SS\cap(\zr)$ at a (Euclidean) distance less than $\epsilon$ from the half-line $\{z + \lambda r\tq \lambda \geq \gamma\}$. Examples of  Dirichlet convex sets are: bounded convex sets, rational polyhedra, closed strictly convex sets, and closed
convex sets whose recession cones are generated by integral vectors (see Corollary 4.6 in~\cite{kocuk2019subadditive}). The following result, a direct consequence from Theorem 2.12 in~\cite{kocuk2019subadditive},  establishes that the integrality gap is finite for convex MIPs whose feasible regions are Dirichlet convex sets.

\begin{prop}\label{prop:Dirichlet_IG}
    Let $P\subseteq \rr^n$ be a {\em Dirichlet convex set}  with respect to $\zr$ and let  $X\subseteq\rr^n$ be a convex set such that $P\cap\int(X)\cap(\zr)\neq \emptyset$. Define $\SS=P\cap X$ and assume that $\valcip(\SS)>-\infty$. Then $\IG(\SS)<+\infty$.
\end{prop}


The following result is useful as a sufficient condition for a convex set to satisfy the Dirichlet property.
\begin{prop}\label{prop:preimage_Dirichlet_convex_sets}
Let $T$ be an injective linear mapping and $\MM$ be  a mixed-integer lattice. Let $X\subseteq\rr^m$ be a Dirichlet convex set w.r.t. $T(\MM)$. Then $T^{-1}(X)$ is a Dirichlet convex set w.r.t.  $\MM$. 
\end{prop}
\begin{proof}
First, by Corollary~8.3.4 in~\cite{rockafellar1970}, we have that $\rec(T^{-1}(X))=T^{-1}(\rec(X))$. Hence, for any $r\in \rec(T^{-1}(X))$, there exists $s\in\rec(X)$ such that $s=T(r)$. Second, since $T$ is injective, we have $T(T^{-1}(\MM))=\MM$. Moreover, since $T(x)\neq 0$ for any $x\neq 0$, we have that $\kappa_T:=\min \frac{\|T(x)\|}{\|x\|}>0$, which implies $\|T(x)\|\geq \kappa_T \|x\|$ for all $x\in\rr^n$.

We must show that for any $\bar x\in T^{-1}(X)\cap \MM$, $\bar r\in \rec(T^{-1}(X))$, $\epsilon>0$ and $\gamma\geq 0$,  there exists $\hat x\in T^{-1}(X)\cap \MM$ at distance less than $\epsilon$ from the half-line $\{\bar x + \lambda \bar r\tq \lambda \geq \gamma\}$.

Since $\bar x\in T^{-1}(X)$ and $\bar r\in \rec(T^{-1}(X))=T^{-1}(\rec(X))$,  we have that there exist $\bar y \in X$ and $\bar s \in \rec(X)$ such that $\bar y=T(\bar x)$ and $\bar s=T(\bar r)$. Moreover, since $\bar x \in \MM$, we have that $\bar y \in X\cap T(\MM)$. 

Now, since $X\subseteq\rr^m$ is a Dirichlet convex set w.r.t. $T(\MM)$, there exists $\hat y\in  X\cap T(\MM)$ at distance less than $\kappa_T\epsilon$ from the half-line $\{\bar y + \lambda \bar s\tq \lambda \geq \gamma\}$, that is, there exists, $\hat \lambda\geq \gamma$ such that 
$$\|\hat y-(\bar y+\hat \lambda \bar s)\|\leq \kappa_T \epsilon.$$ 
Since $\hat y \in X\cap T(\MM)$ and $T^{-1}(X\cap T(\MM))=T^{-1}(X)\cap T^{-1}(T(\MM))=T^{-1}(X)\cap \MM$, there exists $\hat x \in T^{-1}(X)\cap \MM$ such that $\hat y =T(\hat x)$. We obtain
$$\|\hat x-(\bar x+\hat \lambda \bar r)\|\leq\frac{1}{\kappa_T}
   \|T(\hat x)-(T(\bar x)+\hat \lambda T(\bar r))\|
   =\frac{1}{\kappa_T}
   \|\hat y-(\bar y+\lambda \bar s)\|
   \le \epsilon,$$   
\noindent which shows that $T^{-1}(X)$ is a Dirichlet convex set w.r.t. $\MM$.
\end{proof}
\begin{cor}\label{cor:conic_sets}
    Let $\GC$ be a regular cone, $A$ be an $m\times n$ integral matrix and $b$ be an integral vector. Consider the set $\SS=\{x\in\rr^n\tq Ax-b\in \GC\}$, which is the feasible region of a conic programming problem defined by the cone $\GC$. Let $T(x)=Ax$ and assume $T$ is injective. Then if $(\GC+b)\cap T(\rr^n)$ is a Dirichlet convex set w.r.t. $\zz^n$, then $\SS$ is also  Dirichlet convex set w.r.t. $\zz^n$.  
\end{cor}

As an application of Corollary~\ref{cor:conic_sets}, if all faces of the cone $\GC$ are of dimension 1, then the set $(\GC+b)$ is either a (translated) cone or a strictly convex set (see Lemma 3.10 in \cite{DM2016}). Thus, by Corollary 4.6 in~\cite{kocuk2019subadditive} $(\GC+b)$ is a Dirichlet convex set w.r.t. $\zz^n$ except in the case it is a cone with at least one extreme ray that cannot be scaled to be an integral vector.


\subsection{Convex MIPs defined by sets with full-dimensional recession cones}
\label{sec:fullDimRecConeCase}
%



Our main result in this section is that  if the recession cone of $\SS$ is a full-dimensional closed convex cone, then we can estimate the integrality gap of the convex MIP~\eqref{eq:convexIP} via an upper bound. Moreover, this upper bound is independent of the optimal solutions of the continuous relaxation~\eqref{eq:convexP} and of the convex MIP~\eqref{eq:convexIP}.  Next, we define a key parameter of a full-dimensional closed convex cone, which we will use to obtain the upper bound. 

\begin{dfn}[\ConeParameter of a full-dimensional closed convex cone]\label{dfn:fatnessParameterCone}
    For a full-dimensional closed convex cone $\GC\subseteq \mathbb{R}^n$, we define its \coneParameter as
    \[
    \Psi_{\GC, \|\cdot\|} = \max_{d\in \GC}  \left\{  \min_{ f \in \GC_*} \{    f^T d  \tq {  \|f\|_*=1}  \}  \tq {    \|d\|=1} \right\}  ,
    \]
     where  $\|x\|_*=\max\{x^Ty\tq \|y\|=1\}$ is the norm dual to $\|\cdot\|$.
\end{dfn}


\begin{figure}[H]

\begin{subfigure}{0.4\textwidth}
    \centering
\begin{tikzpicture}
   \fill[fill=gray!20] (-2,4)--(0,0)--(2,4);
\draw (-2,4) node[anchor=west, xshift=1.2mm,yshift=-1.4mm]{$\GC_1$}
  --  (0,0) node[anchor=north]{$0$}
  --  (2,4) node[anchor=east]{}
;
\draw [line width=.5mm, dashed,black ] (0,2.683) -- node[above, sloped]{$r=1$}
   (-1.0733,2.1466) node[anchor=south]{}
;    
\draw[fill=none](0,2.683) circle (1.2) node [black,yshift=-1.5cm] {};
\draw [line width=1mm, black ] (0,0) node[anchor=north]{}  --  (0,2.683) node[anchor=north west, yshift=-7.5mm]{}
;
\draw [black, -stealth ] (0,1.3415) node[anchor=north]{}  ->  (1.3415,1.3415) node[anchor= west,yshift=-0.4mm ]{$\frac{1}{\Psi_{\GC_1, \|\cdot\|}}$}
;
\end{tikzpicture}
\caption{A ``narrow" cone $\GC_1$.}
\end{subfigure}
\begin{subfigure}{0.49\textwidth}
    \centering
\begin{tikzpicture}
   \fill[fill=gray!20] (-3,4)--(0,0)--(3,4);
\draw (-3,4) node[anchor=west, xshift=1.2mm,yshift=-1.4mm]{$\GC_2$}
  --  (0,0) node[anchor=north]{$0$}
  --  (3,4) node[anchor=east]{}
;

\draw [line width=.5mm, dashed,black ] (0,2) -- node[above, sloped]{$r=1$}
   (-0.96,1.28) node[anchor=south]{}
;    
\draw[fill=none](0,2) circle (1.2) node [black,yshift=-1.5cm] {};

\draw [line width=1mm, black ] (0,0) node[anchor=north]{}
  --  (0,2) node[anchor=north west, yshift=-7.5mm]{}
;
\draw [black, -stealth ] (0,1) node[anchor=north]{}  ->  (1.3415,1) node[anchor= west,yshift=-0.4mm ]{$\frac{1}{\Psi_{\GC_2, \|\cdot\|}}$}
;
\end{tikzpicture}
\caption{A ``wide" cone $\GC_2$.}
\end{subfigure}
    \caption{A geometric illustration of the \coneParameter using two cones. The reciprocal of the length of the thick line segment is equal to the \coneParameterNoSpace.}
    \label{fig:fatness}
\end{figure}

We provide the intuition behind the \coneParameter by comparing two full-dimensional closed convex cones $\GC_1$ and $\GC_2$ in $\RR^2$ as shown in Figure~\ref{fig:fatness}. Suppose that  we would like to put a unit ball inside these cones that is as close as possible to the origin. As can be seen from Figure~\ref{fig:fatness}, we can put the unit ball in $\GC_2$ closer to the origin compared to $\GC_1$ since $\GC_1$ is ``wider" than  cone $\GC_2$.  In fact, one can show that $\Psi_{\GC_1, \|\cdot\|} < \Psi_{\GC_2, \|\cdot\|}$. Notice that the distance of the center of the unit ball to the origin is inversely related to the \coneParameterNoSpace. Lemma~\ref{lem:ballInFullDimCone} formalizes this intuition.

\begin{lem}\label{lem:ballInFullDimCone}
    Let $\GC\subseteq \mathbb{R}^n$ be a full-dimensional closed convex cone and $r>0$. {Consider the \coneParameter of $\GC$ as defined in Definition~\ref{dfn:fatnessParameterCone}.} Then, the maximizer is attained at an interior point $ d^*$ of $\GC$, $\Psi_{\GC, \|\cdot\|} > 0 $ 
    and   $B( \theta d^*, r)\subseteq \GC$ for all $\theta \ge \frac{r}{\Psi_{\GC, \|\cdot\|}}$. 
\end{lem}
\begin{proof} In order to prove the assertion of the lemma, we analyze two cases. For the first case, we let $d\in\int(\GC)$ such that $\|d\|=1$ and show that $\theta d + r \epsilon \in \GC$ for every $\epsilon \in \mathbb{R}^n$ such that $\| \epsilon \|  = 1$, where $\theta \ge r / \psi_d$ with $ \psi_d = {\min \{   f^T d  : f \in \GC_*, \|f\|_*=1 \} }$. In fact, we have that
\begin{align*}
   \theta d + r \epsilon  \in \GC \  \forall \epsilon: \|\epsilon\| =1 \iff      &  f^T(\theta d + r \epsilon) \ge 0 \quad \forall \epsilon: \|\epsilon\| =1, \forall f\in \GC_*: \|f\|_* =1  \\
  \iff      &  \theta f^T d + r \min\{ f^T \epsilon : \|\epsilon\| =1  \}\ge 0, \    \forall f\in \GC_*: \|f\|_* =1  \\
    \iff      &  \theta f^T d - r  \ge 0, \    \forall f\in \GC_*: \|f\|_* =1  \\
    \iff   &     \theta   \ge \frac{r}{\psi_d}. 
 \end{align*}
 In the first line, we use the fact that $\GC=(\GC_*)_*$ (recall that  $\GC$ is closed) and in third line, we use the fact   that $\min\{ f^T \epsilon :  \|\epsilon\| =1  \} =  - \| f \|_* = -1$. Notice that $ \psi_d > 0$ since $d \in \int(\GC)$ and $f\in \GC_*$.

For the second case, if $d \in \bd (\GC)$, then there exists $f\in \GC_*$ such that $f^T d =0$. In this case, the inner minimization would give an optimal value of 0. Hence, we conclude that the optimizer of the outer maximization should be in the interior of cone $\GC$.

 Finally, we let $\Psi_{\GC, \|\cdot\|} = \max\{\psi_d : d\in \GC, \|d\|=1\} > 0$.
\end{proof}

Now, we provide the main result of this section as stated in the theorem below, which gives an integrality gap bound when $\SS$ has a full-dimensional, closed and convex recession cone as a function of its wideness parameter. 
\begin{thm}\label{thm:proximityInFullDimCone}
    Consider a convex set $\SS \subseteq \rr^n$ and assume that its recession cone $\GC = \rec(\SS)$ is full-dimensional, closed and convex.
    Let  $\hat x \in \SS$. 
    Then, there exists $x' \in \SS \cap(\zr)$ such that $\|   \hat x - x' \|_2  \le  \frac{\sqrt{n_1}}2 \left ( \frac1{\Psi_{\GC, \|\cdot\|_2}} +1 \right )$. 
    Moreover,
    assuming that  $\hat z = \inf_{x\in \SS} \alpha^T x $ is bounded below for $\alpha\in\rr^n$, we have
    \[
    \IG(\SS) = \min_{x\in \SS \cap(\zr)} \alpha^T x -\hat z  \leq  \| \alpha\|_2 \frac{\sqrt{n_1}}2 \left ( \frac1{\Psi_{\GC, \|\cdot\|_2}} +1 \right ).
    \]
\end{thm}

\begin{proof}
We will use Lemma~\ref{lem:ballInFullDimCone} to prove the first statement. Let $d^*$ be the maximizer as defined in this lemma. Then, we have that $\hat x + B(\theta d^*, r) \subseteq K$ for $\theta \ge\frac{r}{\Psi_{\GC, \|\cdot\|}}$. By choosing $r = \sqrt{n_1}/2$ and $\theta=\frac{r}{\Psi_{\GC, \|\cdot\|}}$, we guarantee that the ball $ B(\hat x +\theta d^*, r) \subseteq K$ contains a vector $x'\in\zr$ since $\sqrt{n_1}/2$ is the covering radius of the mixed-lattice $\zr$. Finally, we compute 
\[
\|\hat x - x'\|_2 
\le
  \|\hat x-(\hat x+\theta d^*)\|_2 
 +\|(\hat x+\theta d^*)-x'\|_2
 \le \|\theta d^*\|_2 + r =  \frac{\sqrt{n}}2 \left ( \frac1{\Psi_{\GC, \|\cdot\|_2}} +1 \right ).
\]

%

For the second statement, we look at two cases: i)
If problem $\inf_{x\in \SS} \alpha^T x$ is solvable, then the upper bound on its integrality gap is obtained as a direct consequence of the upper bound on proximity due to the Cauchy-Schwarz inequality.
ii) 
If problem $ \inf_{x\in \SS} \alpha^T x$ is bounded below but not solvable, then we proceed as follows: Notice that for all $\epsilon > 0$, there exists $\hat x_\epsilon\in\SS$ such that $\alpha^T \hat x_\epsilon - \epsilon \le \hat z$, and there exists $x'_\epsilon\in\SS\cap(\zr)$ such that  $\| x'_\epsilon - \hat x_\epsilon \|_2 \le \frac{\sqrt{n_1}}2 \left ( \frac1{\Psi_{\GC, \|\cdot\|_2}} +1 \right )  $, which implies that $\alpha^T ( x'_\epsilon - \hat x_\epsilon ) \le \| \alpha \|_2 \frac{\sqrt{n_1}}2 \left ( \frac1{\Psi_{\GC, \|\cdot\|_2}} +1 \right )$. Therefore, we obtain that
\[
\IG(\SS) = 
\min_{x\in \SS \cap(\zr)} \alpha^T x - \hat z 
\le \alpha^T x_\epsilon' - \hat z \le 
\|\alpha\|_2 \frac{\sqrt{n_1}}2 \left ( \frac1{\Psi_{\GC, \|\cdot\|_2}} +1 \right ) +\epsilon ,
\]
for all $\epsilon > 0$. 
The result follows by letting $\epsilon \to 0^+$. 
\end{proof}



Note that the constant $\Psi_{\GC, \|\cdot\|}$  depends only on $\GC$. In particular, for a conic set $\SS=\{x\in\rr^n\tq \matriz{A}x- \rhs\in {\GC}\}$  whose recession cone $\rec(\SS)=\{x\in\rr^n\tq \matriz{A}x\in {\GC}\}$ is full-dimensional, this constant does not depend on the r.h.s. $\rhs$ of $\SS$, and therefore, one can obtain estimations for the integrality gap that do not depend on the r.h.s. of the conic set, as in the case of linear IPs (Theorem~\ref{thm:cook1986sensitivity}).

\begin{exm}\label{ex:hyperbolaExample}
    Let us consider the hyperbola $\cH=\{x\in\rr_+^2 : \ x_1x_2\ge 1\}$  and
    assume that $\alpha\in\int(\rr_+^2)$. Let us compute a bound on $$\IG(\cH) = \min_{x\in\cH\cap\zz^2} \alpha^Tx - \min_{x\in\cH } \alpha^Tx. $$  For this purpose, we first observe that $\rec(\cH) = \rr_+^2$, which is a full-dimensional closed convex cone, and its \coneParameter is
    \[
    \Psi_{\rr_+^2, \|\cdot\|_2} = \max_{d\in \rr_+^2}  \left\{  \min_{ f \in \rr_+^2} \{    f^T d  \tq {  \|f\|_2=1}  \}  \tq {    \|d\|_2=1} \right\} 
    = \max_{d\in \rr_+^2}  \{\min\{ d_1, d_2  \}:  \|d\|_2=1\} = \frac{1}{\sqrt{2}}.
    \]
    Then, the bound from Theorem~\ref{thm:proximityInFullDimCone} gives
    \[
    \IG(\cH) \le \|\alpha\|_2 \frac{\sqrt{2}}{2}\left(\frac{1}{1/\sqrt{2}}+1\right) =  \|\alpha\|_2 \left(1+\frac{\sqrt{2}}2\right).
    \]
\end{exm}





{We will show below that when $\rec(\SS)$ is not full-dimensional, it may not be possible to quantify the integrality gap only as a function of the recession cone, even when $\SS$ is a rational polyhedral cone. 
\begin{exm}
    Consider the family of polyhedral sets $\SS_k=\{ x\in[0,1]\times \rr : \ (2k-1)x_1 + {x_2} \ge k-\frac12, (2k-1)x_1 - {x_2} \le k-\frac12  \}$, where $k\in\zz_+$. Notice that $\rec(\SS_k)= \{x\in\rr^2: x_1=0, x_2\ge 0\}$ for each $k\in\zz_+$. Suppose that $\alpha=e_2$. Then, the optimal solution for the continuous relaxation is $(\frac{1}{2},0)$ with $\valcp(\SS_k)=0$ and an optimal solution for the integer program is $(0,k)$ with $\valcip(\SS_k)=k$.  Hence, $\IG(\SS_k)=k$. 
    However, since $\rec(\SS_k)= \{x\in\rr^2: x_1=0, x_2\ge 0\}$ is independent of $k$, it is not possible to obtain this value solely as a function of the recession cone. 
\end{exm}
}

\subsection{Convex IPs defined by sets with  polyhedral integer hull}


 In this section, we study convex IPs ($n_2=0$) defined by convex sets that can be approximated by a sequence of polyhedra having rational polyhedral recession cones. We show that for these sets, their integer hulls are rational polyhedra and that the integrality gaps of the associated optimization problems~\eqref{eq:convexIP} are finite. Recall that in Example~\ref{exSOCr}, the integrality gap is infinite, even when the recession cone of the set $\SS$ is a rational ray and the integer hull of $\SS$ is a rational polyhedral set. This illustrates that when the recession cone of $\SS$ is not full-dimensional, one needs stronger conditions for the integrality gap to be finite. 

 Our main result in this section is a characterization of convex sets that can be approximated by a rational polyhedron. The study of this class of sets is motivated by the following remark.
 \begin{rem} \label{rem:approx-S-P}
 If 
 $\PP=\{x\in\rr^n\tq \matriz{A}x\geq b\}$ is a rational polyhedral outer-approximation of a convex set $\SS \subseteq \rr^n$ such that $\SS\cap \zz^n=\PP\cap \zz^n$, then we note that 
\begin{align*}
 \IG(\SS)&=\inf\{\alpha^Tx\tq x\in \SS\cap \zz^n\}-\inf\{\alpha^Tx\tq x\in \SS\} 
 \leq  \inf\{\alpha^Tx\tq x\in \PP\cap \zz^n\}-\inf\{\alpha^Tx\tq x\in \PP\} =\IG(\PP). 
\end{align*} 
\noindent In particular,  we can estimate the integrality gap of~\eqref{eq:convexIP} as follows: $\IG(\SS) \leq\IG(\PP)\leq\|\alpha\|_1 n\Delta_{n-1}(A)$ due to Theorem~\ref{thm:cook1986sensitivity}. 
\end{rem}
However, finding such a polyhedral approximation is not always possible, not even when the integer hull of $\SS$ is a rational polyhedron. For instance, for the hyperbola $\cH$ in Example~\ref{ex:hyperbolaExample}, no such approximation exists since any such polyhedron $\PP$ must contain the $x$-axis, that is an asymptote of $\cH$, and therefore $\PP\cap \zz^n\supsetneq \SS\cap \zz^n$ (for a rigorous proof see Proposition 3 in~\cite{santana2017some}). This example motivates the following definition.
\begin{dfn}\label{dfn:almostrationalpoly}
   A convex set $\SS$ is called {\em almost rational polyhedral} if there exists a polyhedron $\PP=\{x\in\rr^n\tq \matriz{A}x\leq b\}$ such that:  
   \begin{enumerate}
    \item $\SS\subseteq \PP$ 
    \item $\rec(\PP)$ is a rational polyhedral cone.
     \item There exists a sequence of polyhedra $\PP^\nu=\{x\in\rr^n\tq \matriz{A}^\nu x\leq b^\nu\}$ such that $\lim_{\nu\to \infty}(A^\nu,b^\nu)=(A,b)$ and, for all $\nu\geq 1$, $\rec(\PP^\nu)=\rec(\PP)$ and $\PP^\nu\cap \zz^n=\SS\cap\zz^n$.
   \end{enumerate}
\end{dfn}

 In~\cite{DM2013}, sufficient and necessary conditions are given  for a convex set $\SS$ to have a polyhedral integer hull. In order to state their results, we need the following definition: A closed convex set  $\SS$ is {\em thin} if $\valcp(\SS)=-\infty$ if and only if there exists $r\in \rec(\SS)$ such that $\alpha^Tr<0$.
\begin{thm}[Theorem 6 in~\cite{DM2013}]\label{thm:poly}
Let $\SS \subseteq \rr^n$ be a closed convex set. If $\SS$ is thin and the recession cone of $\SS$ is a rational polyhedral cone, then $\textup{conv}(\SS \cap \mathbb{Z}^n)$ is a polyhedron. Moreover, if $\textup{int}(\SS)\cap\zz^n\neq\emptyset$ and $\textup{conv}(\SS \cap \mathbb{Z}^n)$ is a polyhedron, then $\SS$ is thin and $\rec(\SS)$ is a rational polyhedral cone.
\end{thm}

Since polyhedra with rational recession cones have rational polyhedral integer hulls due to Theorem~\ref{thm:poly}, if $\SS$ is almost rational polyhedral, then its integer hull is a rational polyhedral set as $\SS\cap\zz^n=\PP^{1}\cap\zz^n$ and $\PP^1$ is a polyhedron with a rational recession cone. 

We need some notation and lemmas.
Let $W$ be an affine subspace and let $S\subseteq W$ be a set of $W$. The subspace orthogonal to $W$ will be denoted as $W^\perp$. A set $K\subseteq W$ is said to be an $S$-free convex set of $W$ if the interior of $K$ with respect to the topology of $\rr^n$ induced on $W$ does not contain points of $S$. 




\begin{lem}[Corollary 4.10 in~\cite{DM2011}]\label{lem:sfree_prop}
    Let $S_i\subseteq \zz^n\cap W$, $i=1,\dots,N$ such that $S_i=\conv(S_i)\cap \zz^n$. Denote $S=\bigcup_{i=1}^NS_i$.  Let $K \subseteq W$ be an $S$-free convex set of $W$. Then there exists an $S$-free polyhedron $B\subseteq W$ such that $K\subseteq B$.
\end{lem}



The first part of the next result is a characterization of almost rational polyhedral  sets and the second part shows that the integrality gap of a convex MIP whose feasible region is an almost rational polyhedral set can be bounded above as in Remark~\ref{rem:approx-S-P} when the associated sequence of polyhedra are defined by the same l.h.s. matrix.

\begin{thm}\label{thm:thin_finite_IG}
Let $\SS\subseteq \rr^n$ be a closed convex set. Then
\begin{enumerate}
    \item If $\SS$ is almost rational polyhedral, then $\SS$ is thin, $\rec(\SS)$ is a rational polyhedral cone and $\SS$ satisfies the finiteness property. Conversely, if $\SS$ satisfies the finiteness property and its integer hull is a rational polyhedron, then $\SS$ is almost rational polyhedral.
    \item Let $\SS$ be almost rational polyhedral and let $\PP=\{x\in\rr^n\tq \matriz{A}x\leq b\}$ and $\PP^\nu=\{x\in\rr^n\tq \matriz{A}^\nu x\leq b^\nu\}$ for $\nu\geq 1$ be rational polyhedra satisfying the properties in the definition of almost rational polyhedral. Assume that $A^\nu=A$ for all $\nu\geq1$. Then,
$$\IG(\SS)\leq n\|\alpha\|_1\Delta_{n-1}(\matriz{A}).$$
\end{enumerate} 
\end{thm}
\begin{proof}\leavevmode
\begin{enumerate}
    \item 

We first assume that $\SS$ is almost rational polyhedral. Let $\PP=\{x\in\rr^n\tq \matriz{A}x\leq b\}$  and $\PP^\nu=\{x\in\rr^n\tq \matriz{A}^\nu x\leq b^\nu\}$ as in the definition of almost rational polyhedral. Then $\SS_I:=\conv(\SS\cap\zz^n)=\conv(\PP^{1}\cap\zz^n)$ and thus $\SS_I$ is a rational polyhedron with the same recession cone as $\PP^1$. In particular, $\rec(\SS_I)=\rec(\PP^{1})=\rec(\PP)$. Moreover, we have that $\SS_I\subseteq\SS\subseteq \PP$, so we obtain that $\rec(\SS)=\rec(\PP)$, a rational polyhedral cone, and by Lemma~13 in~\cite{DM2013}, we conclude that $\SS$ is a thin set. To see that $\SS$ satisfies the finitenes property, note that if $\valcp(\SS)>-\infty$, since $\SS\subseteq\PP$, we have $\valcp(\PP)>-\infty$. This implies that there exists $r\in\rec(\PP)$ such that $\alpha^Tr<0$. Since $\rec(\SS_I)=\rec(\PP)$, we conclude that $\valcip(\SS)>-\infty$.


Now, we assume that $\SS$ satisfies the finiteness property and its integer hull is a rational polyhedron. Let $\hat \QQ=\{x\in\rr^n\tq a_i^Tx\leq \hat b_i\ \text{for all}\ i=1,\ldots, m\}$ be the rational polyhedron representing the integer hull of $\SS$. Without loss of generality, we assume that for all $i=1,\ldots,m$, we have $a_i\in\zz^n$ and that its components have greatest common divisor equal to 1 (this implies that, for each $z\in\zz$, the equation $a_i^Tx=k$ always has a solution $x\in\zz^n$).
Since $\SS$ satisfies the finiteness property, we obtain that  $b_i:=\sup\{a_i^Tx\tq x\in\SS\}<+\infty$ for all $i=1,\ldots,m$. Define the polyhedron
$$\QQ=\{x\in\rr^n\tq a_i^Tx\leq b_i\ \text{for all}\ i=1,\ldots, m\}.$$
%


 Note that for each $z\in (\QQ\cap\zz^n)\setminus\hat\QQ$, there exists $i=1,\ldots, m$ such that $\hat b_i+1\leq a_i^Tz\leq b_i$. Let $W=\aff(\SS)$. For each $i=1,\ldots,m$, define $S_i=\{x\in\zz^n\tq \hat b_i+1\leq a_i^Tx\leq b_i,\ x\in W\}$ and let $S=\bigcup_{i=1}^mS_i$. Note that $S_i=\conv(S_i)\cap \zz^n$ since $\conv(S_i)\subseteq \{x\in\rr^n\tq \hat b_i+1\leq a_i^Tx\leq b_i,\ x\in W\}$. Also, $\ri(\SS)\cap S_i=\emptyset$ since $\SS\cap S_i=\hat Q\cap S_i\subseteq \{x\in\rr^n\tq \hat b_i +1\leq a_i^Tx\leq \hat b_i\}=\emptyset$. Therefore, $\SS$ is an $S$-free convex set of $W$ and by Lemma~\ref{lem:sfree_prop}, we obtain that there exists an $S$-free polyhedron $B'\subseteq W$ such that $\SS\subseteq B'$. Note that $\ri(B')\neq \emptyset$. Let $B\subseteq \rr^n$ be the polyhedron defined as $B=B'+W^\perp$ and note that $\int(B)=\ri(B')+W^\perp$ so $\int(B)\cap S=\emptyset$. 

Define $\PP=\QQ\cap B$. Then clearly, $\SS\subseteq \PP$. Moreover, since $\rec(\hat \QQ)\subseteq \SS\subseteq\rec(\PP)\subseteq \rec(\QQ)$ and $\rec(\hat \QQ)=\rec(\QQ)$, we obtain that $\rec(\PP)=\rec(\SS)=\rec(\hat \QQ)$ is a rational polyhedral cone.

Next we will construct a sequence of polyhedra $\{\PP^\nu\}_{\nu=1}^\infty$ such that $\lim_{\nu\to \infty}\PP^\nu=\PP$ and, for all $\nu\geq 1$, $\rec(\PP^\nu)=\rec(\PP)$ and $\PP^\nu\cap \zz^n=\SS\cap\zz^n$.

Since $\SS\subseteq \PP$, we have that $\SS\cap\zz^n\subseteq \PP\cap\zz^n$. If $\SS\cap\zz^n= \PP\cap\zz^n$, then we define $\PP^\nu=\PP$ for all $\nu\geq 1$, and thus $\SS$ is almost rational polyhedral. If $\SS\cap\zz^n\neq \PP\cap\zz^n$, by definition of $\PP$, we have that $[(\PP\cap\zz^n)\setminus (\SS\cap\zz^n)]\cap \ri(\PP)=\emptyset$ since $(\PP\cap\zz^n\setminus \SS\cap\zz^n)\subseteq S$, $B$ is an $S$-free set and $\ri(\PP)=\ri(\QQ)\cap \int(B)$. Since $B$ is a polyhedron, we can write $B=\{x\in\rr^n\tq a_i^Tx\leq b_i\ \text{for all}\ i=m+1,\ldots, m+p\}$ for some $p$ appropriate inequalities. With this notation, we have $\PP=\{x\in\rr^n\tq a_i^Tx\leq b_i\ \text{for all}\ i=1,\ldots, m+p\}$. For any $\nu\geq 1$, we define  the following polyhedron:
\begin{equation}\label{eq:Pnu_def}
\PP^\nu =\left\{x\in\rr^n\tq a_i^Tx\leq b_i,\ a_i^Tx+\frac{1}{\nu}a_j^Tx\leq b_i+\frac{1}{\nu}\hat b_j\ \text{for all}\ i=1,\ldots, m+p,\ j=1,\ldots, m\right\}.
\end{equation}

Note that with this notation, we can write 
\begin{equation}\label{eq:P_def}
\PP=\left\{x\in\rr^n\tq a_i^Tx\leq b_i,\ a_i^Tx\leq b_i\ \text{for all}\ i=1,\ldots, m+p,\ j=1,\ldots, m\right\}.
\end{equation}
First we show that $\SS\cap\zz^n= \PP^\nu\cap \zz^n$. Note that all the inequalities defining $\PP^\nu$ are valid for $\SS\cap \zz^n$ since they are either valid for $\hat \QQ\supseteq (\SS\cap\zz^n)$ or for  $\PP\supseteq\SS$, hence $\SS\cap\zz^n\subseteq \PP^\nu\cap \zz^n$. Now, assume for a contradiction that there exists $z\in (\PP^\nu\cap\zz^n)\setminus (\SS\cap \zz^n)$. Since $\PP^\nu\subseteq \PP$, we have that $z\in \PP\setminus\ri(\PP)$ and thus, there exists $i^*=1,\ldots,m+p$ such that $a_{i^*}^Tz=b_{i^*}$ (an inequality of $\PP$ satisfied at equality by this point). Moreover, since $z\notin (\SS\cap \zz^n)$, there exists $j^*=1,\ldots,m$ such that $a^T_{j^*}z\geq \hat b_{j^*}+1$ (an inequality of $\hat \QQ$ that is violated by this point). We obtain that 
$$a_{i^*}^Tz+\frac{1}{\nu}a_{j^*}^Tz\geq b_{i^*}+\frac{1}{\nu}\hat b_{j^*}+\frac{1}{\nu}>b_{i^*}+\frac{1}{\nu}\hat b_{j^*},$$
\noindent  which implies that $z\notin \PP^\nu$, a contradiction. We conclude that $\SS\cap\zz^n= \PP^\nu\cap \zz^n$.

Since $\SS\cap\zz^n\subseteq \PP^\nu$ we have that $\hat Q\subseteq \PP^\nu\subseteq \PP$ and we obtain $\rec(\PP^\nu)=\rec(\PP)$ for all $\nu\geq 1$. On the other hand, from \eqref{eq:Pnu_def} and \eqref{eq:P_def}, it is easy to see that  the vectors and number defining the inequalities of $\PP^{\nu}$ converge to the ones defining $\PP$.

\item Now assume that $\SS$ is almost rational polyhedral, and  let $\PP$  and $\PP^\nu$ the polyhedral sets in the definition of almost rational polyhedral. Assume that for  all $\nu\geq1$, we have $A^\nu=A$. Then, for any $\nu\geq 1$, we have:
$$
\IG(\SS)
    \leq \inf_{x\in\PP^\nu\cap\zz^n}\alpha^Tx-\inf_{x\in\PP}\alpha^Tx
    \leq \IG(\PP^\nu)+\left(\inf_{x\in\PP^\nu}\alpha^Tx-\inf_{x\in\PP}\alpha^Tx\right)
     \leq n\|\alpha\|_1\Delta_{n-1}\left(\matriz{A}\right)+\|\alpha \|_1 n\Delta_{n-1}(A)\frac{1}{\nu}.
$$
Here, the first inequality uses that $\SS\subseteq \PP$ and that $\SS\cap\zz^n=\PP^\nu\cap \zz^n$ for all $\nu\geq 1$, and the third inequality is due to Theorem 1 and Theorem 5 in~\cite{cook1986sensitivity}. Since this chain of inequalities is valid for all $\nu\geq1$, we conclude that $\IG(\SS)\leq n\|\alpha\|_1\Delta_{n-1}(\matriz{A})$ as desired.
\end{enumerate}
\end{proof}






\begin{exm}\label{ex:hyperbolaExampleRevisited}
    Let us consider the hyperbola $\cH=\{x\in\rr_+^2 : \ x_1x_2\ge 1\}$ and
    assume that $\alpha\in\int(\rr_+^2)$. Then $\cH$ is an almost rational polyhedral set. Indeed, it is easy to check that $\PP=\{x\in\rr \tq \ x_1\geq 0,\ x_2\ge 0\}$ and $\PP^\nu=\{x\in\rr \tq \ x_1\geq 1/\nu,\ x_2\ge 1/\nu\}$ for $\nu\geq1$ satisfy the definition. Moreover, all these polyhedra have the same l.h.s. matrix, so we can use part (ii.) in Theorem~\ref{thm:thin_finite_IG} to obtain the following bound
 $$\IG(\cH) \leq  2\|\alpha\|_1$$
Note that, in general, this bound is greater than or equal to the bound from Theorem~\ref{thm:proximityInFullDimCone} obtained for Example~\ref{ex:hyperbolaExample}.
\end{exm}


\section{Integrality gap via polyhedral approximations}
\label{sec:intGapViaPolyApprox}

In virtue of Remark~\ref{rem:approx-S-P}, one might be tempted to use a \textit{rational} polyhedral outer-approximation $\PP$ of the convex set $\SS$ with the property that $\PP\cap \zz^n= \SS \cap \zz^n$
to obtain an upper bound on the integrality gap.
However, as we will argue in the rest of this section, this approach might lead to loose upper bounds, that is, the upper bound obtained for $\IG(\PP)$ due to Theorem~\ref{thm:cook1986sensitivity} might be much larger compared to $\IG(\SS)$, even for the {simple} convex set $\SS = \cB_R := B(0,R) \subseteq \rr^n$.
We proceed as follows:
In Section~\ref{sec:tightBoundB_R}, we compute a tight upper bound of $\|\alpha\|_1$ for $\IG(\cB_R)$, which is independent of~$R$. 
In
Section~\ref{sec:upperBoundB_r-extended}, we show that the rational polyhedral approximation $\PP$ of  $\cB_R$ obtained from~\cite{kocuk2021rational} in an extended space may yield weak bounds even when $n=2$ and $R=1$. 
Finally, in 
Section~\ref{sec:upperBoundB_r-original}, we show that any outer-approximating polytope $\PP=\{x\in\rr^2 \tq Ax \le b\}$ of $\cB_R\subseteq\rr^2$ with $A\in\{-N,\dots,N\}^{m\times 2}$ and $b\in\rr^m$ should have $N\approx\frac{R}{2}$, implying  that the bound for $\IG(\PP)$ obtained from Theorem~\ref{thm:cook1986sensitivity} will be in the order of $\|\alpha\|_1 R$, so it depends linearly on $R$.


\subsection{A tight upper bound for $\IG(\cB_R)$}
\label{sec:tightBoundB_R}

We first present the following proposition that computes an upper bound on the integrality gap  of the convex MIP associated with $ \cB_R \subseteq \rr^n$.
\begin{prop}\label{prop:Proximity-ball}
    Consider the convex set $\cB_R \subseteq \rr^n$ with $R>0$. Then,
    \[
    \IG(\cB_R) = \min_{x\in \SS \cap \zz^n} \alpha^T x - \min_{x\in \SS } \alpha^T x    \leq  \| \alpha\|_1 .
    \]
    \end{prop}

\begin{proof}
    Suppose that $\alpha\neq0$ (since, otherwise, the statement is trivially true). Note that the optimal solution to the continuous relaxation $\min_{x\in \SS } \alpha^T x$ is $\hat x = \frac{\alpha}{\|\alpha\|_2}R$. Now, we construct an integral solution $x' \in \zz^n$ as follows:
    \[
    x'_j = \begin{cases}
        \lfloor \hat x_j \rfloor & \text{ if }  \hat x_j \ge 0 \\
     \lceil \hat x_j \rceil & \text{ if }  \hat x_j < 0 \\
    \end{cases} \ , \quad j=1,\dots,n.
    \]
    Notice that $\|x'\|_2 \le \|\hat x \|_2 = R$. Since we conclude that  $x' \in \cB_R\cap \zz^n$, we deduce that $x'$ is a feasible solution to the convex IP. Moreover, by construction, we have that $\|x' - \hat x \|_\infty \le 1$. Finally, we obtain that 
    \[
    \IG(\cB_R) \le \alpha^T(x'-\hat x) \le \|\alpha\|_1\|x'-\hat x\|_\infty \le  \|\alpha\|_1,
    \]
    where the second inequality follows due to the Hölder's inequality. 
\end{proof}
    Below, we give an example for which the bound  from Proposition~\ref{prop:Proximity-ball} is tight.

\begin{exm} \label{ex:bound-is-tight}
Consider the convex set $\cB_R \subseteq \rr^n$ with $R = 1-\epsilon$ for some small but positive $\epsilon$, which  contains the origin as the  only integral vector. 
Consider the   convex IP 
$\min\{x_1 \tq {x \in \cB_R\cap\mathbb{Z}^n}\} $ and its continuous relaxation 
$\min\{x_1 \tq {x \in \cB_R }\} $. Note the optimal solution to the convex IP is $x^*=0$ with $\valcip(\cB_R)=0$ while the optimal solution to the continuous relaxation IP is $\hat x=(1-\epsilon)e_1$ with $\valcp(\cB_R)=-(1-\epsilon)$.  Therefore, the integrality gap is computed as $\IG(\cB_R) = 1 - \epsilon$, which converges to the bound $\|e_1\|_\infty=1$  from Proposition~\ref{prop:Proximity-ball} as $\epsilon\to0^+$.

    \end{exm}

\subsection{Upper bounds for $\IG(\cB_R)$ via polyhedral outer approximations in an extended space}
\label{sec:upperBoundB_r-extended}

We now focus on a simple case in which $n=2$ and show that the upper bound obtained for $\IG(\PP)$ for an outer-approximating polytope $\PP$ via Theorem~\ref{thm:cook1986sensitivity} can be much larger than $\IG(\cB_R)$.

The polyhedral outer-approximation of a circle in 2D is a well-studied problem in geometry, and it is proven that any such approximation $\PP$ that guarantees a $(1+\delta)$ accuracy, that is, 
$\cB_R \subseteq \PP \subseteq (1+\delta) \cB_R$
should have exponentially many inequalities in the original space~\cite{ball1997elementary,vielma2008lifted,hijazi2014outer,lubin2018polyhedral}.
However, it is possible to construct a compact-size extended formulation as shown in~\cite{ben2001polyhedral}. 
Since this formulation contains irrational entries in $A$ and $b$, we cannot directly use it in our analysis, though.
Recently, a \textit{rational} polyhedral approximation is proposed in~\cite{kocuk2021rational} that retains the same approximation accuracy and   size of the extended formulation as in~\cite{ben2001polyhedral}, which we will use in the sequel (note that $\delta\le\sqrt{1+1/R^2}-1$ must be selected for our purposes so that  $\PP\cap \zz^n= \cB_R \cap \zz^n$). 

As an illustration, we will simplify the setting further and assume that $R=1$. In this case, the outer-approximation procedure proposed in Section~2 of~\cite{kocuk2021rational} gives the following extended formulation (here, $(x_1,x_2)$ are the original variables and $(\xi^j,\eta^j)$, $j=0,1,2$ are the additional variables):
\begin{align*}
x_1 - \xi^0 &\leq 0, & -x_1 - \xi^0 &\leq 0, & x_2 - \eta^0 &\leq 0, & -x_2 - \eta^0 &\leq 0, \\
3\xi^0 + 4\eta^0 - 5\xi^1 &\leq 0, & -3\xi^0 - 4\eta^0 + 5\xi^1 &\leq 0, & -4\xi^0 + 3\eta^0 - 5\eta^1 &\leq 0, & 4\xi^0 - 3\eta^0 - 5\eta^1 &\leq 0, \\
4\xi^1 + 3\eta^1 - 5\xi^2 &\leq 0, & -4\xi^1 - 3\eta^1 + 5\xi^2 &\leq 0, & -3\xi^1 + 4\eta^1 - 5\eta^2 &\leq 0, & 3\xi^1 - 4\eta^1 - 5\eta^2 &\leq 0, \\
\xi^2 &\leq 1, & -3\xi^2 + 4\eta^2 &\leq 0.
\end{align*}
For this representation, the constraint matrix $A$ is $14\times8$ and  $\Delta_{7}(A)=37500$. Therefore, an upper bound for $\IG(\PP)$ is obtained as $2\times37500\times\|\alpha\|_1=75000\|\alpha\|_1$ due to~\cite{paat2020distances} whereas an upper bound for $\IG(\cB_R)$ is obtained as $\|\alpha\|_1$ due to Proposition~\ref{prop:Proximity-ball}, which is much better. 

When we project this polyhedron onto the original variables $(x_1,x_2)$, we obtain the following set of inequalities:
\begin{align*}
x_2 \leq 1, \
- x_2 \leq 1, \
-24x_1 - 7x_2 \leq 25, \
-24x_1 + 7x_2 \leq 25, \
24x_1 - 7x_2 \leq 25, \
24x_1 + 7x_2 \leq 25.
\end{align*}
For this representation, the constraint matrix $A$ is $6\times2$ and  $\Delta_{1}(A)=24$. An upper bound for $\IG(\PP)$ is obtained as $2\times24\times\|\alpha\|_1=48\|\alpha\|_1$ due to~\cite{cook1986sensitivity,paat2020distances}, which is again worse than the bound derived in Proposition~\ref{prop:Proximity-ball} as    $\|\alpha\|_1$ due to .

When we repeat the above procedure for $R=2$, we obtain upper bounds of $7.44\times10^{10}\|\alpha\|_1$   and $4680\|\alpha\|_1
$ for the extended formulation and the projection, respectively. However, notice that our bound of  $\|\alpha\|_1$ remains valid for any $R>0$. This experiment clearly shows that the rational polyhedral outer-approximations from~\cite{kocuk2021rational} are not useful for integrality gap calculations as they are.

\begin{rem}\label{rem:KocukIsNotOptimal}
We should point out that the approximations proposed in~\cite{kocuk2021rational} do not aim to minimize $\Delta_{n-1}$ and it is possible to construct other approximations with much better  $\Delta_{n-1}$ values. For instance, for our running example with $n=2$ and $R=1$, the following is a valid outer-approximation:
\begin{align*}
x_1 \leq 1, \
- x_1 \leq 1, \
x_2 \leq 1, \
- x_2 \leq 1, \
 x_1 + x_2 \leq \frac{1}{\sqrt{2}}, \
 x_1 - x_2 \leq \frac{1}{\sqrt{2}}, \
- x_1 + x_2 \leq \frac{1}{\sqrt{2}}, \
- x_1 - x_2 \leq \frac{1}{\sqrt{2}} .
\end{align*}
We remind the reader that Theorem~\ref{thm:cook1986sensitivity} only requires the $A$ matrix to be integer (there are no restrictions on the $b$ vector). Clearly, for this outer-approximation, $\Delta_1=1$. Although the bound obtained in this case, which is $2\|\alpha\|_1$, is much better than the other polyhedral outer-approximations constructed above, it is still dominated by the bound obtained in Proposition~\ref{prop:Proximity-ball}.
\end{rem}

\subsection{Upper bounds for $\IG(\cB_R)$ via polyhedral outer approximations on the original space}
\label{sec:upperBoundB_r-original}

The observation in Remark~\ref{rem:KocukIsNotOptimal} leads to the following question: 
Given an integer radius $R>0$, what would be the smallest integrality gap upper bound we can obtain from \textit{any} rational polyhedral outer-approximation $\PP=\{x\in\rr^2 \tq Ax \le b\}$ of $\cB_R$ satisfying $\PP\cap \zz^n= \cB_R \cap \zz^n$? In this section, we prove the following theorem: 
\begin{thm}\label{thm:noBoundedCoeffPolytopeCanApproxBall}
 Let $R\in\zz_{++}$ and let $\PP=\{x\in\rr^2 \tq Ax \le b\}$ with $A\in\zz^{m\times 2}$ and $b\in\rr^m$. Then, there exists a matrix $A$ such that $\cB_R\subseteq\PP$, $\PP \cap \zz^2 =  \cB_R \cap \zz^2$ and $\Delta_1(A)= \lceil \frac{R}{2} \rceil$, and thus, $\IG(\cB_R)\leq \IG(\PP) \le  2 \lceil \frac{R}{2} \rceil\ \|\alpha\|_1$.  Moreover, if $\Delta_1(A)<\lceil \frac{R}{2} \rceil$, we have that $\PP \cap \zz^2 \neq  \cB_R \cap \zz^2$.
\end{thm}
In order to prove this theorem, there are two related questions we need to answer:   Given a polyhedral outer-approximation $\PP$ of $\cB_R $ in which  the largest coefficient in~$A$ is at most~$N$, that is, $\Delta_1(A)\leq N$:
i)   what would be the smallest~$N$ such that the outer-approximation~$\PP$
satisfies $\PP\cap \zz^n= \cB_R \cap \zz^n$, and
ii)   what would be the largest~$N$ such that the outer-approximation~$\PP$
does not satisfy $\PP\cap \zz^n= \cB_R \cap \zz^n$?  
%
%
 We answer both of these questions in the following proposition:
\begin{prop}\label{prop:PQ_RNpolytopeProperty}
Let $R, N\in\zz_{++}$ 
\begin{enumerate}
    \item  
   Consider the polytope $\PP_{R,N}$ defined as
   \begin{equation}\label{eq:define-Q_RNpolytope}
\PP_{R,N} := \left\{ x\in\rr^2  \tq  a_1 x_1 + a_2x_2 \le R\|a\|_2, \ a\in\{-N,\dots,N\}^2 \textup{ s.t. either $|a_1|=N$ or $|a_2|=N$}\right\}.       
   \end{equation}
 Then, for $N \ge \lceil\frac{R}{2} \rceil $, we have that  $\PP_{R,N} \cap \zz^2 =  \cB_R \cap \zz^2 $. 
    \item Consider the polytope $\QQ_{R,N}$ defined as
   \begin{equation*}\label{eq:define-P_RNpolytope}
\QQ_{R,N} := \left\{ x\in\rr^2  \tq  a_1 x_1 + a_2x_2 \le R\|a\|_2, \ a\in\{-N,\dots,N\}^2 \right\}.       
   \end{equation*}
 Then, for $N \le \lfloor \frac{R}{2} - \frac{1}{2R} \rfloor  $, we have that  $\QQ_{R,N} \cap \zz^2 \neq  \cB_R \cap \zz^2 $. 
\end{enumerate}
\end{prop}
\begin{proof}\leavevmode
\begin{enumerate}
    \item 
%
Notice that all the inequalities defining $\PP_{R,N} $ are supporting hyperplanes for the ball $\cB_R$. Therefore, $\cB_R \subseteq \PP_{R,N}$. 
Our proof strategy to prove the assertion of this lemma is to characterize all the extreme points of the polytope $\PP_{R,N}$ and show that their norms are less than $\sqrt{R^2+1}$ if $N \ge \lceil\frac{R}{2} \rceil $. This will guarantee that $\PP_{R,N} \cap \zz^2 =  \cB_R \cap \zz^2 $ and the result will follow.

Since the polytope $\PP_{R,N}$ is symmetric with respect to the origin, we will only consider the extreme points in $\RR_+^2$. In addition, the polytope is also symmetric with respect to the line $x_1=x_2$, therefore, we will only consider the extreme point in $\RR_+^2$ such that $x_1 \ge x_2$. Note that the inequalities intersecting in this region have coefficients of the form $(N,k)$, $k=0,\dots,N$. In fact, it is easy to see geometrically that the extreme points 
are located at the intersection of the inequalities $Nx_1+kx_2=R\sqrt{N^2+k^2}$ and $Nx_1+(k+1)x_2=R\sqrt{N^2+(k+1)^2}$, for $k=0,\dots,N-1$, and given as
\[
\hat x_1^{(k)}=\frac{R}{N}\left(
(k+1)\sqrt{N^2+k^2}-k\sqrt{N^2+(k+1)^2}\right), \ 
\hat x_2^{(k)} = R \left( \sqrt{N^2+(k+1)^2} - \sqrt{N^2+k^2} \right).
\]
A straightforward calculation yields that the squared Euclidean  norm of this point is calculated as 
$$
\|\hat x^{(k)}\|^2 = 2\frac{R^2}{N^2} \left [ (N^2+k^2)(N^2+(k+1)^2)-(N^2+k(k+1) )\sqrt{N^2+k^2}\sqrt{N^2+(k+1)^2}  \right].
$$
 We now claim that $
\|\hat x^{(k)}\|^2$ is decreasing in $k$, therefore, the largest size is obtained at $k=0$.  
 To simplify the notation, let us define 
 $a_k =N^2+k^2 $, $b_k=N^2+(k+1)^2$ and $c_k=N^2+k(k+1) $. Then, we have $a_k b_k-c_k^2=N^2$, and we obtain
 \begin{equation*}
     \begin{split}
\|\hat x^{(k)}\|^2 
&= 2\frac{R^2}{N^2} \left [ a_k b_k - \sqrt{a_kb_k-N^2} \sqrt{a_kb_k } \right ] 
= 2\frac{R^2}{N^2} \frac{ a_k^2 b_k^2 - ({a_kb_k-N^2}) {a_kb_k } }{a_k b_k + \sqrt{a_kb_k-N^2} \sqrt{a_kb_k } }\\
&= 2{R^2} \frac{  {a_kb_k } }{a_k b_k + \sqrt{a_kb_k-N^2} \sqrt{a_kb_k } }
= 2{R^2} \frac{  1 }{1 + \sqrt{1-\frac{N^2}{{a_kb_k }}}  } .
     \end{split}
 \end{equation*}
Since $a_kb_k$ is increasing in $k$, we conclude that $\|\hat x^{(k)} \|^2 $ is decreasing in $k$.

 Hence, we deduce that one of the ``largest" extreme points (in size) of $\PP_{R,N}$ is at the intersection of $x_1=R$ and $Nx_1+x_2=R\sqrt{N^2+1}$. We compute this extreme point as
 \[
 x^{(0)}=(R, R(\sqrt{N^2+1}-N)),
 \]
 whose squared Euclidean norm is ${R^2 + R^2(\sqrt{N^2+1}-N)^2}$. In order to make sure that $\PP_{R,N} \cap \zz^2 =  \cB_R \cap \zz^2 $, we need to guarantee that this quantity  is less than ${R^2+1}$. In other words, we should have
 \[
 R^2(\sqrt{N^2+1}-N)^2 < 1 \iff \sqrt{N^2+1}-N \le \frac{1}{R} \iff
 \frac1{\sqrt{N^2+1}+N} \le \frac{1}{R}  \iff R \le N + \sqrt{N^2+1}.
 \]
 Notice that by choosing $N \ge \lceil\frac{R}{2} \rceil $, we immediately satisfy this condition. Hence, the result follows. 

\item
Notice that all the inequalities defining $\QQ_{R,N} $ are supporting hyperplanes for the ball $\cB_R$. Therefore, $\cB_R \subseteq \QQ_{R,N}$. 
    We will now prove that for any positive integer $N$ such that $N \le \lfloor \frac{R}{2} - \frac{1}{2R} \rfloor  $, the integral point $x=(R,1)\in \QQ_{R,N} \setminus  \cB_R $. 
    For this purpose, we will show that
    \[
    a_1 R + a_2 \le R \|a\|_2 , \ \ a\in \{-N,\dots,N\}^2 .
    \]
    Let us first look at cases where $a\in\zz_+^2$.
    \begin{enumerate}
        \item  Consider $(a_1,a_2)=(N,1)$. Observe that we have
    \[
    N R + 1 \le R \sqrt{N^2+1} \iff 2NR+1\le R^2 \Leftarrow N \le \left\lfloor \frac{R}{2} - \frac{1}{2R} \right\rfloor,
    \]
    therefore, the  point $(R,1)$ satisfies the inequality with $(a_1,a_2)=(N,1)$.
    \item Consider any inequality defined by $(a_1,a_2)$ such that $a_1\ge0$ and $a_2 > 0$. Observe that we have
    \[
    a_1 R + a_2 \le R \sqrt{a_1^2+a_2^2} \iff 2\frac{a_1}{a_2}R+1\le R^2 \Leftarrow 2NR+1\le R^2,
    \]
    where the last implication follows due to Item (a).
    \item Consider the inequality defined by $(1,0)$, which is equivalent to $  x_1 \le R$. Notice that this inequality is satisfied by the point $(R,1)$.
    \end{enumerate}
    Now,  
        consider any inequality defined by $(a_1,a_2)$ such that either $a_1 \le  0$ or $a_2 \le 0$. Since the point $(R,1)$ satisfies the inequality with $(|a_1|,|a_2|)$ due to Items (a)-(c), and we have 
        \[
        a_1 R + a_2 \le R \sqrt{a_1^2+a_2^2}   \Leftarrow |a_1| R + |a_2| \le R ,
        \]  
        the result follows. 
\end{enumerate}
\end{proof}








We are now ready to prove Theorem~\ref{thm:noBoundedCoeffPolytopeCanApproxBall}.
\begin{proof}[Proof of Theorem~\ref{thm:noBoundedCoeffPolytopeCanApproxBall}]
    Notice that  due to Proposition~\ref{prop:PQ_RNpolytopeProperty}(i), the polytope $\PP:=\PP_{R,N}$ with    $N =\lceil \frac{R}{2} \rceil  $ is such that $\PP_{R,N} \cap \zz^2 =  \cB_R \cap \zz^2 $ and $\Delta_1=\lceil \frac{R}{2} \rceil$. By computing an upper bound on $\IG(\PP)$ via Theorem~\ref{thm:cook1986sensitivity}, we obtain $ \IG(\PP) \le 2N\|\alpha\|_1 \le 2 \lceil \frac{R}{2} \rceil\|\alpha\|_1$. 
    Now assume that $\Delta_1(A)<\lceil \frac{R}{2} \rceil$. Then, since we have that $\lceil \frac{R}{2} \rceil=\lfloor \frac{R}{2} - \frac{1}{2R} \rfloor+1$ for any integer $R\geq 1$, we obtain by  Proposition~\ref{prop:PQ_RNpolytopeProperty}(ii) that $\PP \cap \zz^2 \neq  \cB_R \cap \zz^2$.
\end{proof}
Note that the upper bound for $\IG(\cB_R)$ obtained in Proposition~\ref{prop:Proximity-ball} is $\|\alpha\|_1$, which is independent of $R$ and significantly better than the bound obtained in Theorem~\ref{thm:noBoundedCoeffPolytopeCanApproxBall}.

\section{Concluding Remarks}

In this paper, we study the integrality gap of convex mixed-integer programs. We study classes of convex sets whose associated optimization problems have finite integrality gap: Dirichlet convex sets, sets with full-dimensional recession cones and sets that can be approximated by polyhedral sets.  We also obtain explicit integrality gap estimations when the feasible region has a full-dimensional recession cone, or it is almost rational polyhedral under mild conditions.
Finally, we explore polyhedral approximations for the feasible region and analyze the possibility of using estimations from the linear mixed-integer  programming literature. However, through a simple example involving the Euclidean ball in $\rr^2$, we show that such estimations can be far off from the integrality gap estimations we derive previously.  

While our results provide new insights into the integrality gap analysis for convex mixed-integer programs, they are not exactly analogous to the results known for linear mixed-integer  programs. For example, in this paper, we have not carried out a \textit{proximity analysis}, that is, we have not quantified 
   the distance between the optimal solutions to the mixed-integer program~\eqref{eq:convexIP} and its continuous relaxation~\eqref{eq:convexP}. This is due to two main reasons: (i.) A general convex mixed-integer program and/or its continuous relaxation may not have an optimal solution, even if the objective function is bounded below on $\SS$, and (ii.) When the feasible region is a general convex set, we do not have the properties of rational polyhedral sets that are very useful in the context of integer programming (for instance, the existence of Hilbert basis for rational polyhedral cones is crucial for Theorem~\ref{thm:cook1986sensitivity}).
   In this respect, we believe that having more structure on the set $\SS$ can lead to stronger results.
A promising future research direction is to study the integrality gap and the proximity of~\eqref{eq:convexIP} for special cases of convex mixed-integer programs, e.g. mixed-integer second-order  cone programs.



\section*{Statements and Declarations}

\subsection*{Competing interests}
The authors declare that they have no competing interests.

\subsection*{Funding}
 We would  like to thank for the support from the  ANID grant  Fondecyt \# 1210348 and TUBITAK visiting researcher grant under Program 2221.






\subsection*{Availability of data and materials}
Our manuscript does not contain any associated data.




\bibliographystyle{plain}
\bibliography{references}

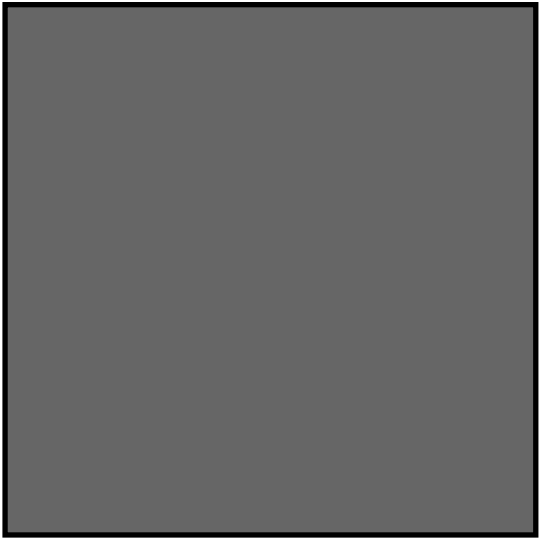
\end{document}